\documentclass[12pt]{amsart}
\usepackage{amsmath}
\usepackage{amssymb}
\usepackage{epsfig}
\usepackage{graphicx}
\numberwithin{equation}{section}

\input xy
\xyoption{all}

\calclayout
\allowdisplaybreaks[3]

\theoremstyle{plain}
\newtheorem{prop}{Proposition}
\newtheorem{thm}[prop]{Theorem}

\theoremstyle{definition}

\newtheorem{conj}[prop]{Conjecture}

\newtheorem*{cond*}{Condition (H1)}

\def\ra{\rightarrow}

\def\bP{{\mathbb P}}

\def\rH{{\mathrm H}}

\def\Pic{\mathrm{Pic}}

\def\lim{\mathrm{lim}}

\makeatother
\makeatletter

\author{Yuri Tschinkel}
\address{Courant Institute\\
                New York University \\
                New York, NY 10012 \\
                USA }
\email{tschinkel@cims.nyu.edu}

\address{Simons Foundation\\
160 Fifth Avenue\\
New York, NY 10010\\
USA}

\author{Kaiqi Yang}
\address{Courant Institute\\
                New York University \\
                New York, NY 10012 \\
                USA }
\email{ky994@nyu.edu}

\title[Stably rational del Pezzo surfaces]{Potentially stably rational del Pezzo surfaces over nonclosed fields}

\begin{document}
\date{\today}

\maketitle

\section{Introduction}

A geometrically rational surface $S$ over a nonclosed field $k$ is $k$-birational to either a del Pezzo surface of degree $n\in [1,\ldots, 9]$ 
or a conic bundle (see \cite{isk}).
Throughout, we assume that $S(k)\neq \emptyset$. 
This implies $k$-rationality of $S$ when $n\in [5,\ldots , 9]$ or when the number of degenerate fibers of the conic bundle
is at most 3. 

Let $G_k$ be the absolute Galois group of $k$, it acts on exceptional curves and on the geometric Picard group  
$\Pic(\bar{S})$ of $S$. The surface $S$ is called {\em split} over $k$ if all exceptional curves are defined over $k$, 
and {\em minimal} if no blow-downs are possible over $k$, i.e.,  
there are no $G_k$-orbits consisting of pairwise disjoint exceptional curves.   
A minimal del Pezzo surface of degree $\le 4$ over $k$ is not rational (see, e.g., \cite[Theorem 3.3.1]{mt}). 
A surface $S$ is called {\em stably rational} over $k$ 
if $S\times \mathbb P^m$ is birational to a projective space, over $k$. 
A necessary condition for stable rationality of $S$ over $k$ is

\begin{cond*}
\label{cond:h1}
\[ 
\label{eqn:h1}
\rH^1(G_{k'}, \Pic(\bar{S}))=0, \quad \text{ for all finite extensions} \quad k'/k. 
\] 
\end{cond*}

\noindent
As a special case of a general conjecture of Colliot-Th\'el\`ene and Sansuc one expects that this is also sufficient:

\begin{conj}
If $S$ satisfies condition (H1) then $S$ is stably rational over $k$. 
\end{conj}

\noindent
Only one example of a minimal, and thus nonrational, but stably rational del Pezzo surface of degree $\le 4$ is known at present \cite{ct1,ct2,bs}; 
in this case the Galois group acts via the symmetric group $\mathfrak S_3$, the smallest nonabelian group (see Section~\ref{sect:3-4} for a description of
this action). Finding another example is a major open problem. 
There are however examples of minimal del Pezzo surfaces of degrees $1\le n\le 4$ and of conic bundles with at least $4$ degenerate fibers,
{\em failing} (H1) and thus stable rationality over $k$. 

For $n=3,2,$ and $1$,  
the Galois group $G_k$ of $k$ acts on the primitive Picard group of $S$ (the orthogonal complement of the canonical class in $\Pic(S)$)
through the Weyl group $W(\mathsf E_{9-n})$; 
for $n=4$ and conic bundles with $n+1$ degenerate fibers through $W(\mathsf D_{n+1})$.
These actions have been extensively studied, in connection with arithmetic applications and rationality questions, e.g., 
the Hasse Principle and Weak Approximation, 
when $k$ is a number field (see e.g., \cite{manin}, \cite{kst}, \cite{swd}, \cite{urabe}, \cite{li}, \cite{dp}).

This note is inspired by a recent result of Colliot-Th\'el\`ene concerning stable rationality 
of geometrically rational surfaces over quasi-finite $k$, i.e., 
perfect fields with procyclic absolute Galois groups \cite{ct-stable}. 
The main result of \cite{ct-stable} is that over such fields, stably rational surfaces
are actually rational. This follows from: 

\begin{thm}\cite[Theorem 4.1]{ct-stable}
\label{thm:1}
Let $S$ be a surface over $k$, 
geometrically rational with $S(k)\neq \emptyset$. If $S$ is split by a cyclic extension 
and is not $k$-rational then there exists a finite separable extension $k'/k$ such that 
$$
\rH^1(G_{k'}, \Pic(\bar{S}))\neq 0.
$$ 
\end{thm}

The proof proceeds via a case-by-case analysis of actions of (conjugacy classes of) elements of the corresponding Weyl groups, investigated in 
connection with the study of the Hasse-Weil zeta function of del Pezzo surfaces.   
For $n=4$ this is due to \cite{swd}, \cite{manin} and also
follows from \cite{kst}; for $n=3$ this goes back to Trepalin.

For general $k$, it is of interest to identify Galois actions potentially giving rise to minimal, stably rational surfaces, 
i.e., those satisfying (H1). 
This has been done in \cite{kst} for del Pezzo surfaces of degree 4. 
Our main result is a classification of the relevant actions in degrees $3$, $2$, and 1. 
In particular, this immediately gives an alternative proof of Theorem~\ref{thm:1} for del Pezzo surfaces; 
there are simply no cyclic groups on the list of actions 
in Sections~\ref{sect:2} and \ref{sect:1}. 

The computation is organized as follows: the {\tt magma} program produces a  list of subgroups  (modulo conjugation); then, starting with small groups, computes
first cohomology groups. When it finds a group with nontrivial first cohomology, it eliminates all groups containing it. 
In this way, the poset of subgroups is rapidly exhausted. After that, minimality and presence of conic bundles are easily checked. 
The code and lists of orbit decompositions for subgroups satisfying (H1) are available at:

\

\centerline{ 
{\tt cims.nyu.edu/\,$\tilde{}$\,tschinke/papers/yuri/18h1dp/magma/ }}

\

\noindent
{\bf Acknowledgments:}  We are grateful to J.-L. Colliot-Th\'el\`ene for helpful comments and suggestions. 
The first author was partially supported by NSF grant 1601912.

\section{Degree 4 and 3}
\label{sect:3-4}

We use the following notation:
\begin{itemize}
\item $\mathfrak C_n$ - cyclic group of order $n$
\item $\mathfrak D_{n}$ - dihedral group of order $2n$
\item $\mathfrak F_n$ - Frobenius group of order $n(n-1)$ 
\item $\mathfrak S_n$ - symmetric group of order $n!$
\end{itemize}

Let $S$ be a minimal del Pezzo surface of degree 4, satisfying Condition (H1). We recall Theorems E and F from \cite{kst}: 

\begin{itemize}
\item If $S$ admits a conic bundle structure then $S$ is $k$-birational to
$$
x^2-ay^2=f_3(t), \quad \deg(f_3)=3, 
$$
where $a=\mathrm{disc}(f_3)$. The Galois group of the splitting field is $\mathfrak S_3$. One of the degenerate fibers, over $\infty$, is defined
over $k$, the other three, corresponding to roots of $f_3$,  are permuted by the $\mathfrak S_3$ action, the components of all 
singular fibers are exchanged the Galois action of the discriminant quadratic
extension. A surface $S$ of this type is not rational but stably rational over $k$. 
\item Assume that $S$ does not admit a conic bundle structure over $k$. 
Let $\tilde{S}\ra S$ be a blowup, with center in a suitable $k$-rational point; $\tilde{S}$ is a 
smooth (nonminimal) cubic surface admitting a conic bundle with 5 degenerate fibers. Then $\tilde{S}$ is of type $I_1,I_2$, or $I_3$  
listed in \cite[Theorem 6.15]{kst}. The Galois groups of corresponding splitting fields are $\mathfrak S_2\times \mathfrak S_3$ in the first case, a nontrivial extension of $\mathfrak S_3$ by $\mathfrak S_2$ in the second case,   
and a nontrivial central extension of $\mathfrak S_2\times \mathfrak S_3$ by 
$\mathfrak S_2$ in the third case.  
In Case 1, there are two degenerate fibers defined over $k$, with nontrivial Galois action on the components of the fibers, and three Galois conjugated degenerate fibers. In the Cases 2 and 3, the Galois-action has two orbits on the set of degenerate fibers, of length 2 and 3.  
\end{itemize}

Our first result is:

\begin{prop}
\label{prop:cubic}
There are no minimal cubic surfaces satisfying Condition (H1). 
In particular, a $k$-minimal cubic surface is not stably rational over $k$. 
\end{prop}

\begin{proof}
Direct calculation with {\tt magma}. 
\end{proof}

\section{Degree 2}
\label{sect:2}

In the description below we encode the Galois action on the set of exceptional curves as follows: we write
$\{ v_1^{r_1}, \ldots, v_m^{r_m}\}$ for the decomposition into orbits, where $v_j$ are dual intersection graphs, enumerated below, and $r_j$ are their multiplicities.
For minimal del Pezzo surfaces of degree 2 we find unique orbit types with cardinality 4, 8, 18, 24, 30, 42, two types of cardinality 2 and 12, and three types of cardinality 6 and 10.
The occurring graphs for each orbit are symmetrical: each vertex has the same number of outgoing edges (with multiplicities). We write 
$$
(n)[s_1^{t_1}, \ldots, s_d^{t_d}]
$$
for a graph with $n$ vertices, where each vertex has $t_j$ outgoing edges of multiplicity $s_j$ (equal to the intersection number between the two exceptional curves connected by this edge). 
The corresponding graphs are listed below:

\begin{itemize}
\item[-] $2_c:=(2)[1]$  $\bullet - \bullet$, $2:=(2)[2]$ $\bullet = \bullet$
\item[-]  $4:=(4)[1^3]$
\item[-]  $6_1:=(6)[1^2,2]$, $6_2:=(6)[1^4]$, $6_c=(6)[1]$ conic bundle
\item[-]  $8:= (8)[1^3,2]$
\item[-] $10_1:=(10)[1^4,2]$, $10_2:=(10)[1^6]$, $10_c=(10)[1] $ conic bundle
\item[-] $12:=(12)[1^5,2]$, $12_c= (12)[1]$, conic bundle
\item[-] $14:=(14)[1^6,2]$
\item[-] $18:= (18)[1^8,2]$
\item[-] $24:=(24)[1^{11},2]$
\item[-] $30:=(30)[1^{14},2]$
\item[-] $42:=(42)[1^{20},2]$
\end{itemize}

In the following propositions we list the 
structure of Galois groups of splitting fields, the structure or orbits on the set of exceptional curves, and the stabilizers
for each orbit.

\begin{prop}
\label{prop:d1}
Assume that $S$ is a minimal degree 2 del Pezzo surface over $k$ satisfying Condition (H1). Then $S$ either admits 
a conic bundle structure over $k$ or is one of the following types, each corresponding to a conjugacy class
of subgroups in $W(\mathsf E_7)$:
\begin{itemize}
\item[dP2(1)] $\mathfrak D_{7}$: $\{ 14^4\}$, trivial stabilizer
\item[dP2(2)] $\mathfrak F_7$: $\{14,42\}$, specializes to $\mathrm{dP2(1)}$, when restricted to $\mathfrak D_7\subset \mathfrak F_7$. 
\item[dP2(3)] $\mathfrak D_{15}$: $\{ 6_1,10_1^2,30\}$, stabilizers
$\{ \mathfrak C_5, \mathfrak C_3, 1\}$.  
\item[dP2(4)] $\mathfrak C_3\rtimes \mathfrak F_5$: $\{6_1,10_2^2,30\}$, 
stabilizers $\{ \mathfrak D_5, \mathfrak C_6,\mathfrak C_2\}$, 
with $\mathfrak C_2$ not normal. 
\end{itemize} 
\end{prop}

Below we list all possible conic bundle types. Each $X$ admits two conic bundle structures over $k$, with isomorphic 
Galois actions on the set of exceptional fibers of the corresponding conic bundle. We organize by cardinalities of orbits on
these sets, and by the orbit structure on the set of exceptional curves of $X$. 

\

\noindent
{\em 3+3:}

\

\begin{itemize}
\item[D6(1)] $\mathfrak S_3$: $\{  2, 6_1^3, 6_2^2, 6_c^4 \}$, stabilizers $\{ \mathfrak C_3, 1, 1, 1\}$
\item[D6(2)] $\mathfrak C_3\rtimes \mathfrak S_3$: $\{2, 6_1^2, 6_c^4,18\}$, stabilizer 
$\{ \mathfrak C_3^2, \mathfrak C_3, \mathfrak C_3, 1 \}$
\end{itemize}

\
  
\noindent
{\em 5+1:}

\

\begin{itemize}
\item[D6(3)] 
$\mathfrak D_5$: $\{ 2_c^2, 2,  10^3_1, 10^2_c \}$, stabilizer $\{ \mathfrak C_5, \mathfrak C_5,1,1\}$
\item[D6(4)]
$\mathfrak F_5$: $\{2^2_c, 2,10_1,10^2_2, 10_c^2\}$, stabilizer $\{ \mathfrak D_5, \mathfrak D_5, \mathfrak C_2,\mathfrak C_2,\mathfrak C_2  \}$; 
$\mathfrak C_2$ is not normal
 \end{itemize}

\noindent
{\em 6:}

\

\begin{itemize}
\item[D6(5)]
$\mathfrak D_6$: $\{2, 6_1,12^2,12_c^2\}$,    stabilizer $\{ \mathfrak C_6, \mathfrak C_2,1,1\}$.    
\item[D6(6)]
$\mathfrak D_6$: $\{2,6_1,6_2^2,12,12_c^2\}$,    stabilizer $\{ \mathfrak S_3, \mathfrak C_2, \mathfrak C_2,1,1\}$.
\item[D6(7)]
$\mathfrak S_4$: $\{2,6_1,12_c^2,24\}$,    stabilizer $\{ \mathfrak A_4, \mathfrak C_2^2, \mathfrak C_2,1\}$.  
\item[D6(8)]
$\mathfrak S_4$: $\{4^2,6_2^2,12,12_c^2\}$,    stabilizer $\{ \mathfrak S_3,\mathfrak C_2^2,\mathfrak C_2, \mathfrak C_2\}$.   
\item[D6(9)]
$\mathfrak S_4$: $\{6_2^2,8,12,12_c^2\}$,    stabilizer $\{ \mathfrak C_4, \mathfrak C_3, \mathfrak C_2, \mathfrak C_2\}$.         
\item[D6(10)]
$\mathfrak S_3^2$: $\{2,12,12_c^2,18\}$,    stabilizer $\{ \mathfrak C_3 \times \mathfrak S_3,\mathfrak C_3,\mathfrak C_3,\mathfrak C_2\}$.
\item[D6(11)]
$\mathfrak C_2\times \mathfrak S_4$: $\{2,6_1,12_c^2,24\}$,    stabilizer $\{ \mathfrak C_2 \times \mathfrak A_4, \mathfrak C_2^3, \mathfrak C_2^2, \mathfrak C_2\}$, 
the stabilizer $\mathfrak C_2$ is not normal, and this case does not reduce to $\mathrm{D6}(7)$, with $\mathfrak S_4$-action.
\item[D6(12)]
$\mathfrak C_2\times \mathfrak S_4$: $\{6_2^2,8,12,12_c^2\}$,    stabilizer $\{ \mathfrak D_4,\mathfrak S_3,\mathfrak C_2^2,\mathfrak C_2^2 \}$.
\item[D6(13)]
$\mathfrak S_5$: $\{2,12_c^2,30\}$,   stabilizer $\{\mathfrak A_5,\mathfrak D_5,\mathfrak C_2^2\}$.
\item[D6(14)]
$\mathfrak S_5$: $\{10_2^2,12,12_c^2\}$,   stabilizer $\{ \mathfrak D_6,\mathfrak D_5,\mathfrak D_5\}$.
\end{itemize}

Some types above are specializations of other types, by restriction to subgroups:

\

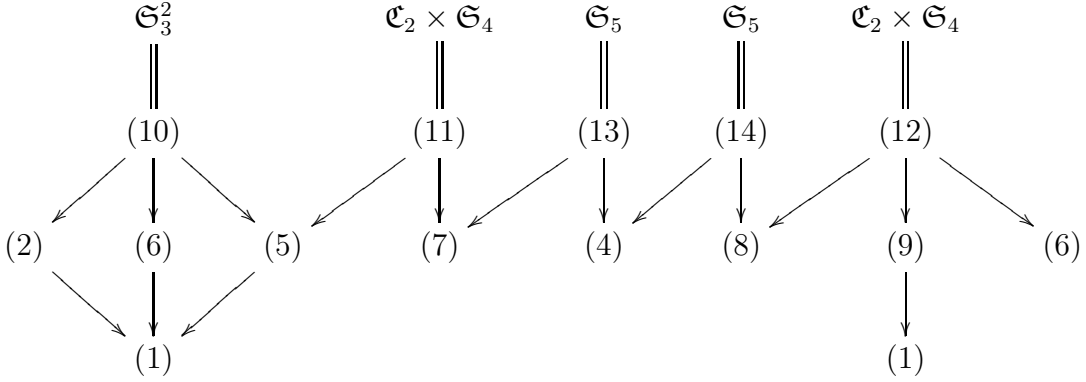
\centerline{
\xymatrix{ 
 &   \mathfrak S_3^2 \ar@{=}[d] &  & \mathfrak C_2\times \mathfrak S_4 \ar@{=}[d] & \mathfrak S_5\ar@{=}[d]  & \mathfrak S_5 \ar@{=}[d] & \mathfrak C_2\times \mathfrak S_4 \ar@{=}[d] & \\
 &   (10) \ar[dl] \ar[d] \ar[dr] &  & (11) \ar[dl] \ar[d] & (13) \ar[dl] \ar[d]  & (14) \ar[dl]\ar[d] & (12) \ar[dl]\ar[d]\ar[dr] & \\
(2) \ar[dr] &  (6)\ar[d] & (5) \ar[dl]  & (7) &  (4) & (8) & (9)\ar[d] & (6)\\   
  & (1) &  &  &  &  & (1) &
   }
}

\

\section{Degree 1}
\label{sect:1}

\begin{prop}
\label{prop:d1}
If $S$ is a minimal degree 1 del Pezzo surface satisfying Condition (H1) then $S$ is a conic bundle over $k$.  
\end{prop}

As Galois orbits we have unions of degenerate fibers of conic bundles ($4_c, 6_c, 8_c, 10_c$) and 
several new orbit types:

\begin{itemize}
\item[-]  $3:=(3)[2^2]$ 
\item[-]  $4_1:=(4)[2^2]$, $4_2:=(4)[1^2,2]$, $4_3=(4)[1^2,3]$, 
\item[-] $5:=(5)[1^2,2^2]$.
\item[-] $6_3:=(6)[2^2,3]$, $6_4:=(6)[1^3,2^2]$
\item[-] $10_3:=(10)[1^3,2^4]$,  $10_4:=(10)[1^4,2^2,3]$, 
\item[-] $12_1:=(12)[1,2^6]$
$12_2:=(12)[1^4,2^3]$, 
$12_3:=(12)[1^2,2^4,3]$, \\
$12_4:=(12)[1^8,2]$,
$12_5:=(12)[1^6,2^2,3]$
\item[-]
$20_1:=(20)[1^2,2^8,3]$,
$20_2:=(20)[1^8,2^4]$,
$20_3:=(20)[1^6,2^6,3]$, 
$20_4:=(20)[1^{12},2^2]$,
$20_5:=(20)[1^9,2^6]$
\item[-] 
$24_1:=(24)[1^2,2^{10},3]$,
$24_2:=(24)[1^{13},2^3]$
\item[-] 
$36_1:=(36)[1^{18},2^5]$, 
$36_2:=(36)[1^{18},2^8,3]$.
\item[-] $40:=(40)[1^{18},2^{10},3]$
\end{itemize}
The types of occurring conic bundles are listed below, each corresponding to a conjugacy class of subgroups in $W(\mathsf E_8)$: 

\

\noindent
{\em 1+3+3:}

\

\begin{itemize}
\item[D7(1)] 
$\mathfrak S_3^2$: $\{ 2^2_c, 3^4,4^2_2,  6_3^2 , 6^4_c,  12_2^4, 12_3^2, 36_1^2, 36_2\}$, \\
stabilizer
$\{\mathfrak C_3 \rtimes \mathfrak S_3,  \mathfrak D_6, \mathfrak C_3^2, \mathfrak S_3, \mathfrak S_3, \mathfrak C_3, \mathfrak C_3, 1,1\}$

\end{itemize}

\

\noindent
{\em 1+1+5:}

\

\begin{itemize}
\item[D7(2)]
$\mathfrak D_{10} $: $\{2^4_c,   4^2_1, 4_3, 5^4, 10_2^2, 10^2_c,  20_1^2, 20_2^4, 20_4^2\}$, stabilizer\\
$\{\mathfrak D_5,\mathfrak C_5,\mathfrak C_5,\mathfrak C_2^2,\mathfrak C_2,\mathfrak C_2,1,1,1 \}$.
\item[D7(3)]
$\mathfrak C_2\times \mathfrak F_5$: $\{ 2^4_c,  4_1^2, 4_3, 
10_1^2, 10_c^2,  20_1^2, 20_3, 20_4^6\}$, stabilizer\\ 
$\{\mathfrak F_5,
\mathfrak D_5,\mathfrak D_5,
\mathfrak C_2^2,\mathfrak C_2^2,
\mathfrak C_2,\mathfrak C_2,\mathfrak C_2\}$. 
\end{itemize}

\noindent
{\em 2+5:}

\

\begin{itemize}
\item[D7(4)] 
$\mathfrak C_5 \rtimes \mathfrak C_4$: $\{ 4_1^2, 4_3, 4_c^2, 5^4, 10_2^2, 10_c^2, 20_2^4, 20_4^2, 20_5^2\}$, stabilizer\\
$\{ \mathfrak C_5,\mathfrak C_5,\mathfrak C_5,\mathfrak C_4,\mathfrak C_2,\mathfrak C_2,1,1,1\}$
\item[D7(5)]
$\mathfrak F_5$: 
$\{4_1^2, 4_3, 4_c^2, 10_1^2, 10_c^2, 20_3, 20_4^6, 20_5^2\}$, stabilizer\\
$\{ \mathfrak C_5,\mathfrak C_5,\mathfrak C_5,\mathfrak C_2,\mathfrak C_2,1,1,1\}$
\item[D7(6)]
$\mathfrak C_5 \rtimes \mathfrak D_4$: 
$\{ 4_1^2, 4_3, 4_c^2,
5^4,
10_2^2, 10_c^2,
20_2^4, 20_4^2,
40\}$, 
stabilizer\\
$\{\mathfrak C_{10},\mathfrak C_{10},\mathfrak D_5,
\mathfrak D_4,
\mathfrak C_2^2,\mathfrak C_2^2, 
\mathfrak C_2,\mathfrak C_2, 1\}$ 
\item[D7(7)]
$\mathfrak C_2 \times \mathfrak F_5$: 
$\{2_c, 4_1^2, 4_3, 4_c^2,
10_1^2, 10_c^2, 
20_3,  20_4^6, 20_5^2\}$, stabilizer\\
$\{\mathfrak F_5, 
\mathfrak D_5,\mathfrak D_5,\mathfrak D_5,
\mathfrak C_2^2,\mathfrak C_2^2,
\mathfrak C_2,\mathfrak C_2,\mathfrak C_2\}$; 
the stabilizer $\mathfrak C_2$ is not normal and we cannot reduce to $\mathrm D7(5) = \mathfrak F_5$
\item[D7(8)]
$\mathfrak C_2^2\rtimes \mathfrak F_5$:
$\{4_1^2, 4_3, 4_c^2,10_1^2, 10_c^2,20_3, 20_4^6,40\}$, stabilizer\\
$\{ \mathfrak D_{10},\mathfrak D_{10},\mathfrak F_5, \mathfrak C_2^3,\mathfrak C_2^3, \mathfrak C_2^2,\mathfrak C_2^2,\mathfrak C_2\}$
\end{itemize}

\noindent
{\em 1+6:}

\

\begin{itemize}
\item[D7(9)]
$(\mathfrak C_3 \rtimes \mathfrak S_3) \rtimes \mathfrak C_2$: 
$\{2_c^2, 4_2^2,  6_2^2,  12_1^2, 12_4^4, 12_5, 12_c^2,  36_1^2, 36_2\}$, stabilizer\\
$\{ \mathfrak C_3 \rtimes \mathfrak S_3,  
\mathfrak C_3^2,
\mathfrak S_3,\mathfrak C_3,\mathfrak C_3,\mathfrak C_3,\mathfrak C_3,1,1\}$
\item[D7(10)]
$\mathfrak S_3 \wr \mathfrak C_2$: 
$\{2_2^2,4_2^2,  6_2^2,  12_5, 12_c^2,24_1, 24_2^2,  36_1^2, 36_2\}$, stabilizer\\
$\{\mathfrak S_3^2,\mathfrak C_3 \times \mathfrak S_3,\mathfrak D_6,\mathfrak S_3,\mathfrak S_3,\mathfrak C_3,\mathfrak C_3,\mathfrak C_2,\mathfrak C_2\}$
\end{itemize}

Again, some types are specializations, by restriction to subgroups: 

\

\centerline{
\xymatrix{ 
\mathfrak C_5 \rtimes \mathfrak D_4 \ar@{=}[d]  & & & \mathfrak C_2^2\rtimes \mathfrak F_5 \ar@{=}[d]& \mathfrak S_3 \wr \mathfrak C_2\ar@{=}[d] & \\
(6) \ar[d] \ar[rd] &       &       & (8) \ar[dl] \ar[d]   & (10) \ar[d] \ar[dr] &  \\
(4)                     & (2) & (3)  & (7) \ar[d]             & (1)                       & (9) \\ 
                         &       &       &  (5)                      &                            &   
}
}

\bibliographystyle{alpha}
\bibliography{h1dp}
\end{document}